\documentclass[11pt]{amsart}
\usepackage{palatino}
\usepackage[all]{xy,xypic}
\usepackage{tikz-cd}
\usepackage{amsfonts,amsmath,oldgerm,amssymb,amscd}

\newcommand{\ol}{\overline}

\newtheorem{theorem}{Theorem}[section]
\newtheorem{proposition}[theorem]{Proposition}
\newtheorem{lemma}[theorem]{Lemma}
\newtheorem{definition}[theorem]{Definition}
\newtheorem{corollary}[theorem]{Corollary}

\newcommand{\ga}{\alpha}	\newcommand{\gb}{\beta}

\newcommand{\bp}{\begin{proposition}}
\newcommand{\ep}{\end{proposition}}
\newcommand{\bl}{\begin{lemma}}
\newcommand{\el}{\end{lemma}}
\newcommand{\bt}{\begin{theorem}}
\newcommand{\et}{\end{theorem}}
\newcommand{\bc}{\begin{corollary}}
\newcommand{\ec}{\end{corollary}}
\newcommand{\bd}{\begin{definition}}
\newcommand{\ed}{\end{definition}}
\newtheorem{conjecture}[theorem]{Conjecture}
\newtheorem{question}[theorem]{Question}

\def\rmk{\refstepcounter{theorem}\paragraph{{\bf Remark} \thetheorem}}
\def\proof{\paragraph{Proof}}

\newcommand{\remark}{\rmk}
\newcommand{\bco}{\begin{conjecture}}
\newcommand{\eco}{\end{conjecture}}
\title{On set-theoretic complete intersections for smooth  curves in three-dimensional affine schemes}
\author{Lisa Mandal} \address{Department of Mathematics and Statistics, Indian Institute of Science Education and Research Kolkata, Mohanpur, 741246, India.}
\email{lm23rs066@iiserkol.ac.in}  
 \author{Md. Ali Zinna} \address{Department of Mathematics and Statistics, Indian Institute of Science Education and Research Kolkata, Mohanpur, 741246, India.}
\email{zinna@iiserkol.ac.in,zinna2012@gmail.com}

\begin{document}
\maketitle

\section{Abstract}
We prove that every local complete intersection curve in $\operatorname{Spec}(A)$, where $A$ is a commutative Noetherian ring of dimension three, is a set-theoretic complete intersection. An analogous result is established for local complete intersection surfaces when $A$ is a four-dimensional affine algebra over the algebraic closure of a finite field of $p$ elements.
Furthermore, we show that any local complete intersection curve  (respectively, surface) in ${\rm Spec}(A)$, where  
$A$ has dimension three (respectively,  four), having trivial  conormal bundle is, in fact,  a complete intersection.

\section{Introduction}
The study of algebraic varieties and their defining equations lies at the heart of algebraic geometry. A fundamental question in this area is to determine when a given variety can be expressed as the set-theoretic intersection of a prescribed number of hypersurfaces. Instead of  asking whether the defining ideal  $I(V)$ can be generated by a specific number of elements, one considers whether there exists an ideal $J$,  generated by the desired number of elements,  whose radical coincides with that of $I(V)$. This naturally leads to the following notion:

An ideal $I$ in a commutative Noetherian ring $A$ is said to be {\it set-theoretically generated} by $r$ elements if there exist  $f_1,\cdots,f_r\in A$ such that  $\sqrt{I}=\sqrt{(f_1,\cdots,f_r)}$. 

The study of set-theoretic generation has its roots in a classical result of Kronecker (\cite{kr}, 1882), who showed  that any ideal $I$ in a commutative Noetherian ring $A$ of dimension $n$ is set-theoretically generated by $n+1 \,\,  (=\text{dim}(A)+1)$ elements. Specifically, for an ideal $I\subset k[x_1,\cdots,x_n]$, where $k$ is a field,  the corresponding variety $V(I)$ can be described as the set-theoretic intersection of $n+1$ hypersurfaces. 

This bound was improved in 1972  by Storch \cite{s} and independently by  Eisenbud and Evans  \cite{ee2}, who proved that any ideal $I\subset A[T]$ is set-theoretically generated by $\text{dim}(A[T])$ elements. Consequently,  the variety $V(I)$ in $\mathbb{A}_k^n$ is the set-theoretic  intersection of $n$ hypersurfaces. 

 In this context, the notion of {\it set-theoretic complete intersection}  becomes central.  An ideal is called a set-theoretic complete intersection if it is set-theoretically generated by a regular sequence whose length equals the height of the ideal. An ideal $I$ in a commutative Noetherian ring $A$ is called a {\it local complete intersection}  if,  for every prime ideal $\mathfrak{p}\supset I$,  $I_{\mathfrak{p}}$ is generated by a regular sequence of length $\text{ht}(I)$.  A variety $V\subset \mathbb{A}_k^n$ is said to be  a local complete intersection if its defining ideal $I(V)$ is a local complete intersection. 
 In 1975, Murthy  \cite[p. 206]{mu0} posed the following fundamental question regarding the relationship between local complete intersection varieties and set-theoretic complete intersections:

\begin{question}
Let  $V\subset \mathbb{A}_k^n$ be a variety which is a local  complete intersection (for example,  if  $V$ is smooth).  Is $V$  a set-theoretic complete intersection?
\end{question}

This question is motivated by a result of Ferrand  \cite{fe}, which asserts  that  any local complete intersection curve $\mathcal{C}\subset \mathbb{A}_k^3$ is a  set-theoretic complete intersection. 
 Bloch-Murthy-Szpiro  \cite{bms}  showed that if  $k$ is the algebraic closure of a finite field, then  any local complete intersection surface $\mathcal{S}\subset \mathbb{A}_k^4$ is a set-theoretic complete intersection.  

For curves, an affirmative answer to Murthy's question was given by Mohan Kumar \cite[Corollary 5]{mk0}, who established that every local complete intersection curve in $\mathbb{A}_k^n$ is a set-theoretic complete intersection. 

In this article, we continue the study of set-theoretic complete intersections in arbitrary commutative Noetherian rings, with particular emphasis on local complete intersection curves in three-dimensional affine schemes.  Our first main result is the following (see Theorem \ref{main1}):

\begin{theorem}\label{main1-int}
Let $A$ be a commutative Noetherian ring of dimension $3$, and  let $I \subset A$ be a local complete intersection ideal of height $2$. Then $I$ is a set-theoretic complete intersection.
\end{theorem}

This result extends Murthy's question to a broader setting, giving an affirmative answer in dimension three. Moreover, our approach yields the following  result. 

\begin{theorem}
 Let $A$ be a commutative Noetherian ring of dimension $\geq 2$.  Suppose $I\subset A$ is a local complete intersection ideal of height $2$ such that $I/I^2$ is a free $A/I$-module of rank $2$.  Then  $I$ is  a  complete intersection.
\end{theorem}

 Motivated by the result of Bloch-Murthy-Szpiro mentioned above, we  extend our investigation to local complete intersection ideals in affine algebras of dimension four over the algebraic closure of a finite field with $p$ elements. In this setting, we establish the following result (see Theorem \ref{main2}):

\begin{theorem}\label{main2-int}
Let $A$ be an affine algebra of dimension $4$ over $\overline{\mathbb{F}}_p$,  and let $I \subset A$ be a local complete intersection ideal of height $2$. Then $I$ is a set-theoretic complete intersection.
\end{theorem}
\medskip

 
 
 Finally, we note  that Theorems 2.2 and 2.4 lead directly to the following corollaries.
 
 \bc
 Let $A$ be a ring of dimension $3$, and let $\mathcal{C}\subset {\rm Spec}(A)$ be a curve which is local complete intersection. Then $\mathcal{C}$ is a set-theoretic complete intersection. 
\ec

\bc
Let $A$ be an affine algebra  of dimemsion $4$ over $\ol{\mathbb{F}}_p$, and let $\mathcal{S}\subset {\rm Spec}(A)$ be a local complete intersection surface. Then $\mathcal{S}$ is a  set-theoretic complete intersection.
\ec

\section{Preliminaries}
\noindent{\bf Assumptions.} Throughout this paper, rings are assumed to be commutative Noetherian and projective modules are finitely generated and of constant rank. For a ring $A$,  $\text{dim}(A)$ will denote the Krull dimension of $A$. 
\smallskip

We begin with the following definition.
\medskip

\begin{definition}
  An ideal $I$ of a ring $A$ is called a \emph{complete intersection ideal} if there exists a regular sequence $f_1,\cdots,f_r$ in $I$ such that $I=( f_1,\cdots,f_r)$, where $\text{ht}(I)=r$.\\
An ideal $I\subset A$ is called a \emph{set-theoretic complete intersection ideal} in $A$, if there exists a regular sequence $f_1,\cdots,f_r$ in $A$ such that $\sqrt{I}=\sqrt{( f_1,\cdots,f_r )}$, where $\text{ht}(I)=r$.\\
 An ideal $I\subset A$ is called a \emph{local complete intersection ideal} in $A$, if $I_{\mathfrak{p}}$ is a complete intersection ideal in  $A_{\mathfrak{p}}$ for any prime ideal $\mathfrak{p}$ that contains $I$.
\end{definition}
\medskip

\rmk If $I\subset A$ is a local complete intersection ideal of height $n$, then $I/I^2$ is a projective $A/I$-module of rank $n$.
\medskip


We now state a lemma that will serve as an important  ingredient in the proof of the main theorem.

\begin{lemma}\label{fact-1}
  Let $I\subset A$ be a local complete intersection ideal of height $n$.  Then $I$ is a complete intersection ideal if and only if it is generated by $n$ elements.
\end{lemma}
\proof
If $I$ is a complete intersection ideal of height $n$, then by definition,  it is generated by a regular sequence of length $n$, and in particular by $n$ elements.

Conversely, suppose $I\subset A$ is a local complete intersection ideal of height $n$,  and suppose $I = (f_1, \dots, f_n)$ is generated by $n$ elements.  We want to show that $I$ is generated by a regular sequence of length $n$.   We proceed by induction.  

Assume  as the induction hypothesis  that we have constructed a regular sequence  $g_1,\cdots, g_k\in I$  and $k < n$  such that  $$I = (g_1, \cdots , g_k, f_{k+1},\cdots , f_n).$$

Let the associated primes of $A/(g_1,\cdots, g_k)$ be $\mathfrak{p}_1, \cdots, \mathfrak{p}_s, \mathfrak{q}_1, \cdots, \mathfrak{q}_t$,   where:\\
$\bullet$ $f_{k+1} \in \mathfrak{p}_i$ for $1 \leq i \leq s$, and \\
$\bullet$ $f_{k+1} \notin \mathfrak{q}_j$ for $1 \leq j \leq t$.   

Since $\text{depth}(A_{\mathfrak{p}_j}) = k < n=\text{ht}(I)$,  and $I$ is a local complete intersection ideal,  it follows that $I\nsubseteq \mathfrak{p}_i$ for any $i$.  Hence,  using prime avoidance,  we can find an element
$$\lambda  \in I\cap \mathfrak{q}_1 \cap \cdots \cap \mathfrak{q}_t\setminus \mathfrak{p}_1 \cup \cdots \cup \mathfrak{p}_s.$$

Because $\lambda \in (f_{k+2},\cdots,f_n)$,  we deduce that  
$$I=(g_1,\cdots, g_{k+1},f_{k+2},\cdots,f_n),$$
where we define $g_{k+1} := f_{k+1} + \lambda$.   By construction, $g_{k+1}$ does not lie in any associated prime of $A/(g_1, \dots, g_k)$.  Thus, its image in  $A/(g_1, \dots, g_k)$ is a non-zerodivisor.  
Consequently,  $g_1, \cdots , g_{k+1}$ form a regular sequence.  
 By induction, we obtain a regular sequence $g_1, \cdots, g_n$ that generates $I$.  Therefore,  $I$ is a complete intersection ideal.
\qed
\medskip

The following result is due to Serre \cite{se},  can also be found in  \cite[Lemma 1.3]{mu0}. Given an $A$-module $M$, we denote by ${\rm hd}_A(M)$ the homological dimension of $M$. 

\bt\label{serre-1}
Let $A$ be a ring and $M$ be a finitely generated $A$-module with {\rm hd}$_A(M)\leq 1$.  Suppose that {\rm Ext}$_A^1(M,A)$ is a cyclic $A$-module.  Then, there is an exact sequence 
$$0\longrightarrow A\longrightarrow P\longrightarrow M\longrightarrow 0,$$
where  $P$ is a projective $A$-module.
\et
\medskip

We recall the following lemma from \cite{ak}; see also \cite[Lemma 1.1]{mu0}.

\bl\label{ext}
Let $A$ be a ring and $I\subset A$ be a local complete intersection ideal of height $r$. Then there is a canonical isomorphism
$$\omega={\rm Ext}_A^{r-1}(I,A)={\rm Ext}_A^r(A/I,A)\simeq {\rm Hom}_{A/I}(\wedge^r(I/I^2),A/I).$$
\el
The following result is due to Ferrand and Szpiro \cite{fe, sz}; a proof can be found in  \cite[Theorem 10.3.3]{ir}.
\begin{theorem} \label{fs}
Let $A$ be a ring,  and let $I\subset A$ be a local complete
intersection ideal of height $r\geq 2$ with $\dim(A/I) \leq 1$.  Then there exists a local complete intersection ideal $J\subset A$ of height $r$ such that:
\begin{enumerate}
  \item $\sqrt{I}=\sqrt{J}$, and
  \item  $J/J^{2}$ is a free $A/J$-module of rank $r$.
\end{enumerate}
\end{theorem}

\begin{theorem}\label{ge}
Let $(A,\mathfrak{m})$ be a local ring, and let $I \subset A$ be a proper  ideal.  Suppose there exist elements $f_{1},\cdots,f_{n} \in I$ such that $I=(f_1,\cdots,f_n)+ I^{2}$. Then $I=(f_{1}, \cdots, f_{n})$.
\end{theorem}

\proof
Since $\mathfrak{m}$ is the unique maximal ideal of $A$, we have $I\subset \mathfrak{m}$. As $A$ is Noetherian, the ideal $I$ is finitely generated. Let $J=(f _1,\cdots,f_n )$  be the ideal generated by these elements. Then $I=J+I^2$. By the Nakayama's Lemma, we conclude that $I=J$. This completes the proof.
\qed
    
\medskip

The following  proposition is from \cite[Proposition 2.1.1]{ir}.

\begin{proposition}{\label{splits}}
Let $A$ be a ring and $P$ be an $A$-module. Then   
 $P$ is projective if and only if every exact sequence 
 $0\longrightarrow M \xrightarrow{\;\,f\;\,} N \xrightarrow{\;\,g\;\,} P \longrightarrow 0$ splits.
\end{proposition} 
\medskip

\begin{definition}
Let $A$ be a ring, and let $M$ be an $A$-module. Let $u \colon M \longrightarrow A$ be an $A$-linear map. For $p \geq 0$, define 
$$
d_u \colon \Lambda^p(M) \longrightarrow \Lambda^{p-1}(M)$$
be the contraction with $u$, i.e., for $e_1, \ldots, e_p \in M$, 
$$
d_u(e_1 \wedge \cdots \wedge e_p)
= \sum_{i=1}^p (-1)^{i+1} u(e_i) \, e_1 \wedge \cdots \wedge \hat{e}_i \wedge \cdots \wedge e_p,$$
where  $\hat{e}_i$ indicates that $e_i$ is omitted.

It is easy to check that
$$
\cdots \xrightarrow{\; d_u \;} \Lambda^p(M) \xrightarrow{\; d_u \;} \Lambda^{p-1}(M) 
\xrightarrow{\; d_u \;} \cdots \xrightarrow{\; d_u \;} \Lambda^1(M) 
\xrightarrow{d_u = u} \Lambda^0(M) \longrightarrow 0$$
is a complex and that for $a \in \Lambda^p(M), \; b \in \Lambda(M)$, we have 
\begin{center}
$d_u(a \wedge b) = d_u(a) \wedge b + (-1)^p a \, d_u(b)$.
\end{center}
The complex defined above is called the {\it Koszul complex}  associated to $u$.
\end{definition}
\medskip

\bl\label{exact}
Let $A$ be a  ring,  and let $I\subset A$ be a complete intersection ideal of height $2$.  Suppose that $I=(x,y)$,  where $x, y$ form  a regular sequence in $A$.  Then the Koszul complex associated to $x$ and $ y$
$$0\longrightarrow A \xrightarrow{\;\, \phi \;\,} A^2 \xrightarrow{\;\, \psi \;\,} I\longrightarrow 0$$
is exact,  where the maps are defined by $\phi(a)=(-ya, xa)$ and $\psi(a_1,a_2)=xa_1+ya_2$. 
\el

\proof
Since $\{x,y \}$ forms a regular sequence,  $x \in A$ is a non-zerodivisor in $A$,  and  $y$ is a non-zerodivisor in $A/(x)$.

Suppose $r \in A$ satisfies $\phi(r) = (-yr,xr) = 0$. 
Then $rx = 0$.  Since $x$ is a non-zerodivisor,  it follows that $r=0$.  Hence $\phi$ is injective.

Let $(a,b) \in A^2$ with  $\psi(a,b) = ax + by = 0$.    
Reducing this equality modulo the ideal $(x)$ gives $\overline{b}\,\overline{y} = \ol 0$ in $A/(x)$.  
Since $\overline{y}$ is a non-zerodivisor in $A/(x)$, we deduce $\overline{b}=\ol 0$; hence $b = xt$ for some $t \in A$.  
 Substituting back, we get
\begin{center}
$ax + (xt)y = x(a+ty) = 0 \;\; \implies \;\; a+ty=0\;\;$ (because $x$ is a non-zerodivisor). 
\end{center}

Thus $a = -ty$ and $b=xt$, so
\[
(a,b) = (-yt, xt) = \phi(t).
\]
Hence $\ker(\psi) \subset \mathrm{im}(\phi)$. The reverse inclusion $\mathrm{im}(\phi) \subset \ker(\psi)$ is immediate from $\psi \circ \phi = 0$. Therefore, $
\ker(\psi) = \mathrm{im}(\phi)$.

By definition, $\psi$ maps $A^2$ onto the ideal $(x,y)$, so $\psi$ is surjective. Therefore, the Koszul complex  
$$ 0 \longrightarrow\; A \xrightarrow{\;\, \phi \;\,} A^2 \xrightarrow{\;\, \psi \;\,} I \;\longrightarrow 0
$$
is exact.
\qed
\medskip

As an immediate consequence of Lemma \ref{exact}, we obtain the following result.
\bl
Let $A$ be a ring and $I\subset A$ be a local complete intersection ideal of height $2$. Then the homological dimension ${\rm hd}(I)\leq 1$.
\el
\smallskip

\section{Main theorems}
This section is devoted to the proofs of our main results. As preparation, we begin with a key structural property of complete intersection ideals of height two in local rings. The following proposition will play a crucial role in the proofs of the theorems.

\bp\label{regular}
Let $A$ be a local ring,  and let $I\subset A$ be a complete intersection ideal of height $2$.  Then any minimal generating set of $I$ consisting of two elements forms a regular sequence in $A$. 
\ep

\proof
Since $I$ is a complete intersection ideal of height $2$,  there exists a regular sequence $\{a,b\}$ generating $I$.  Let $\{c,d\}$ be another minimal generating set of $I$.  We want to show $\{c,d\}$ also forms a regular sequence. 

Since both sets generate $I$,  there exist elements  $r_1,\,r_2,\,s_1,\,s_2\in A$ such that
\[
a = cr_{1} + dr_{2} 
\qquad \text{and} \qquad 
b = cs_{1} + ds_{2}.
\]
Equivalently,
\[
\begin{pmatrix} a \\ b \end{pmatrix} 
= U \begin{pmatrix} c \\ d \end{pmatrix},
\qquad \text{where } \quad 
U = \begin{pmatrix} r_{1} & r_{2} \\ s_{1} & s_{2} \end{pmatrix} \in M_{2}(A).
\]
We now show that $U$ is invertible. 
Suppose, for contradiction, that $\det(U) \in \mathfrak{m}$, where $\mathfrak{m}$ is the unique maximal ideal of $A$.  
Reducing modulo $\mathfrak{m}$, we obtain elements
\[
\overline{a},\; \overline{b}, \;\overline{c}, \;\overline{d}\; \in I/\mathfrak{m}I.
\]
Since $I/\mathfrak{m}I$ is a $2$-dimensional vector space over the field $A/\mathfrak{m}$,  
both $\{\overline{a},\overline{b}\}$ and $\{\overline{c},\overline{d}\}$ are bases.  
But then the change-of-basis matrix is $\overline{U}$, the reduction of $U$ modulo $\mathfrak{m}$. 
If $\det(U) \in \mathfrak{m}$, then $\det(\overline{U}) = 0$, contradicting the fact that  $\ol U$ must be invertible. 
Hence $\det(U) \notin \mathfrak{m}$, and therefore  $U \in GL_{2}(A)$.  

As in the standard Koszul description of a height two complete intersection ideal, consider the following Koszul complexes:
\[
0\longrightarrow A \xrightarrow{\;\, \alpha\;\, } A^{2}\xrightarrow{\;\, g\;\, } I \longrightarrow 0
\]
\[
0\longrightarrow A \xrightarrow{\;\, \beta\;\, } A^{2}\xrightarrow{\;\, h \;\, } I \longrightarrow 0,
\]

where
\[
\alpha(1) = \begin{pmatrix}-b \\ a\end{pmatrix}, \quad g(e_{1}) = a,\ g(e_{2}) = b,
\]
and
\[
\beta(1) = \begin{pmatrix}-d \\ c\end{pmatrix}, \quad h(e_{1}) = c,\ h(e_{2}) = d.
\]
Note that $M:=U^t:A^2\longrightarrow A^2$ induces an $A$-module homomorphism $\sigma:A\xrightarrow{\; \wedge^2(M)\; } A$.  
It is  straightforward to verify that the following diagram is commutative:

\[
\begin{tikzcd}
0 \arrow{r} & A \arrow{r}{\alpha} \arrow{d}{\sigma} & A^{2} \arrow{r}{g} \arrow{d}{M} & I \arrow{r} \arrow{d}{\mathrm{id}} & 0\\
0 \arrow{r} & A \arrow{r}{\beta} & A^{2} \arrow{r}{h} & I \arrow{r} & 0
\end{tikzcd}
\]
That is, $\beta \circ \sigma = M\circ \alpha$ and  $ g = h\circ M$. 

Since $\{a, b\}$ is a regular sequence,  Lemma \ref{ext} implies that  the top Koszul complex in the diagram is exact.  As $M$ is invertible, the induced vertical map 
$A^{2}\longrightarrow A^{2}$ in the diagram is an automorphism of the free module $A^2$. Moreover,  $\sigma=\wedge^2(M):A\longrightarrow A$ is an automorphism; indeed,  $\wedge^2(M)$ acts by multiplication with ${\rm det}(M)$, which is a unit  $u\in A^*$.  We now proceed to  prove that $\{c,d\}$ forms a regular sequence in $A$.

First we show that $c$ is a non-zerodivisor. 
 Suppose $cx=0$. Then  
$$ h(x,0)=cx+0y=0.$$  
Since $M$ is an automorphism, there exists $(x_1,x_2)\in A^2$ such that  $M(x_1,x_2)=(x,0)$.   
Thus  from the commutativity of the above diagram, we obtain  $$ g(x_1,x_2)=(h\circ M)(x_1,x_2)=h(x,0)=0.$$  
Since the top Koszul complex in the diagram is exact, $(x_1,x_2)\in \text{im}(\ga)= \text{ker}(g)$, so there exists $z\in A$ with  $ \ga(z)=(x_1,x_2)$.   
Again from the commutativity of the above diagram   
$$ (\gb\circ \sigma)(z)=(M\circ \ga)(z)=(x,0). $$  
Therefore,  
$$ \beta(uz)=(x,0)\quad\implies\quad (-duz,\,cuz)=(x,0).$$  
Thus $x=-duz$ and $cuz=0$. In particular, $cz=0$.  
 Repeating the same reasoning with $cz=0$, we obtain inductively  
$$ x=(-1)^n d^n v z_n \quad\text{and}\quad cz_n=0 ,$$  
for some unit $v$ and $z_n\in A$. This shows that $x\in \cap(d^n)$.  By the Krull's Intersection Theorem \cite[Corollary  8.25]{sh}, it follows that $x=0$.  Hence $c$ is a non-zerodivisor.  

Next, we show that  $d$ is a non-zerodivisor modulo  $(c)$.  
 Let tilde denote reduction modulo $(c)$.  Suppose that $\widetilde{d}\,\widetilde{y} = \widetilde{0} \quad \text{in } A/(c)$. Then $dy = ct$ for some $t \in A$, implies that $h(-t,y)=0$. Since $M$ is an automorphism, there exists $(y_1,y_2)\in A^2$ such that  
$$ M(y_1,y_2)=(-t, y).$$  
From the commutative diagram, we have
$$ g(y_1,y_2)=(h\circ M)(y_1,y_2)=0.$$  
Since the top Koszul complex in the diagram is exact, $(y_1,y_2)\in {\rm ker}(g)=\text{im}(f)$, so there exists $w\in A$ with  $ \ga(w)=(y_1,y_2)$.   
Applying the commutativity relation again,
$$ (\gb\circ \sigma)(w)=(M\circ \ga)(w)=M(y_1,y_2)=(-t,y). $$  
Therefore,  
$$ \gb(uw)=(-t,y)\quad\implies\quad (-duw,\,cuw)=(-t,y).$$  
In particular, $cuw=y$,  which gives $\widetilde{y}=\widetilde 0$ in $A/(c)$. Therefore, $d$ is a non-zerodivisor in $A/(c)$.  It follows that 
$\{c,d\}$ forms a regular sequence in $A$ and the proof is complete.
\qed

\medskip

After setting up the required background, we are now ready to present the main results.  The following theorem generalizes  Ferrand's result \cite{fe} (see also \cite[Corollary 1.8]{mu0}) to the most general framework, namely, that of three-dimensional commutative Noetherian rings.

\begin{theorem}\label{main1}
  Let $A$ be a ring of dimemsion $3$, and let $I\subset A$ be a local complete intersection  ideal of height $2$. Then $I$ is a set-theoretic complete intersection.
\end{theorem}
\begin{proof}
The proof proceeds in several steps.\\
  {\bf Step-1:} Since $I$ is a local complete intersection ideal of height $2$ with $\text{dim}(A/I)\leq 1$, Theorem \ref{fs} ensures the existence of a local complete intersection ideal $J\subset A$ of height $2$ such that : 
   \begin{enumerate}
       \item[(a)] $\sqrt{I} = \sqrt{J}$,  and
       \item[(b)] $J/J^2$ is a free $A/J$-module of rank $2$.
       \end{enumerate}
     Since $J/J^2$ is free, by Lemma \ref{ext}, it follows that 
     ${\rm Ext}_A^1(J,A) \simeq A/J$. Moreover, we note that ${\rm hd}(J) \leq 1$.  Therefore, by Theorem \ref{serre-1}, there exists an exact sequence 
     \begin{equation}
     0\longrightarrow A\xrightarrow{\;\, f \;\,} P\xrightarrow{\;\,g\;\,} J\longrightarrow 0,   
     \end{equation}
where $P$ is a projective $A$-module of rank $2$. 

 Since $g$ is  $A$-linear, we can define the associated Koszul complex:
\begin{equation}
0\longrightarrow \wedge^2(P)\xrightarrow{\;\, h \;\,} P\xrightarrow{\;\,g\;\,} J\longrightarrow 0,
\end{equation}
 where the map $h$ is defined by 
 \begin{center}
 $h(p\wedge q)=g(p)q-g(q)p$, \,\,\,\,  for  $p, q \in P$,
 \end{center}
  and $g$ is as in (1).
  \medskip

{\bf Step-2:}
In this step, we show that the projective $A$-module $P$, obtained in Step 1,  has trivial determinant; that is, $\wedge^2(P)\simeq A$. 
 
  Let $\mathfrak{p}\subset A$ be a prime ideal  containing $I$.  Localizing the Koszul complex (2) at $\mathfrak{p}$, we obtain the following complex 
 \begin{equation}
 0\longrightarrow (\wedge^2(P))_\mathfrak{p} \xrightarrow{\;h_\mathfrak{p}\;} P_\mathfrak{p}\xrightarrow{\;g_{\mathfrak{p}}\;} J_{\mathfrak{p}}\longrightarrow 0 .
 \end{equation}
 
 Since $(\wedge^2(P))_\mathfrak{p}$ and $P_\mathfrak{p}$ are free $A_\mathfrak{p}$-modules of rank $1$ and $2$,  respectively,  we may  identify the complex (3) with  
 \begin{equation}
 0\longrightarrow A_{\mathfrak{p}} \xrightarrow{\;h_\mathfrak{p}\;} A_{\mathfrak{p}}^2\xrightarrow{\;g_{\mathfrak{p}}\;}  J_{\mathfrak{p}}\longrightarrow 0.
 \end{equation}
 Since $J$ is a local complete intersection of height $2$, there exists a regular sequence $\{a, b\}$ in $A_{\mathfrak{p}}$ such that  $J_{\mathfrak{p}}$ is generated by $a$ and $b$. 
  Let $\{e_1, e_2\}$ be the standard basis of $A_{\mathfrak{p}}^2$. Then  
$$h_{\mathfrak{p}}(x)=(-g_{\mathfrak{p}}(e_2)x, g_{\mathfrak{p}} (e_1)x).$$

Since $g_{\mathfrak{p}}$ is surjective, by Proposition \ref{regular}, we may further assume that  
$g_{\mathfrak{p}}(e_1)=a$\, and \, $g_{\mathfrak{p}}(e_2)=b$. Applying Lemma \ref{exact}, we conclude that the sequence  
$$ 0\longrightarrow A_{\mathfrak{p}} \xrightarrow{\;h_\mathfrak{p}\;}  A_{\mathfrak{p}}^2\xrightarrow{\;g_{\mathfrak{p}}\;}  J_{\mathfrak{p}}\longrightarrow 0$$
is exact. 

Since the Koszul complex (2) becomes exact after localization at every prime ideal, and exactness of a complex of finitely generated modules is a local property, the Koszul complex (2) is exact. 
Finally, combining (1) and (2), we deduce that  $$\wedge^2(P)\simeq {\rm ker}(P\rightarrow J)\simeq  A.$$ 

{\bf Step-3:} 
Since $J/J^2$ is a free $A/J$-module of rank $2$, we may write 
$$J/J^2 \simeq \ol{A}\ol{c}\oplus \ol{A}\ol{d},$$ 
 where $c, d\in J$.  By Lemma \ref{ge}, $J$ is locally generated by $c$ and  $d$. It then follows from  Proposition \ref{regular} that $\{c,d\}$ forms a regular sequence locally. Consequently, $J/(c)$  is  locally free of rank one as an $A/(c)$-module. In other words, $J/(c)$ is a projective $A/(c)$-module of rank one. 
 
Tensoring the exact sequence (1) with $A/(c)$, we obtain   
$$A/(c) \xrightarrow{\;\, \widetilde h \;\,} P/cP\xrightarrow{\;\, \widetilde{g} \;\,} J/(c)\longrightarrow0 ,$$
where tilde denotes reduction modulo $(c)$.  

Since $J/(c)$ is a projective $A/(c)$-module , Proposition \ref{splits}    yields a decomposition $$P/cP\simeq  J/(c) \oplus  \text{im}(\widetilde h) .$$ Taking exterior powers, we  obtain 
  $$\wedge^2(P)\otimes A/(c)\simeq \wedge^2(P/cP)\simeq  J/(c) \otimes \text{im}(\widetilde h). $$ 
 Since $\wedge^2(P)\simeq A$ (shown in Step 2), it follows that
 $$A/(c)\simeq J/(c) \otimes \text{im}(\widetilde h),$$  and hence  $$J/(c) =  \text{im}(\widetilde h)^*.$$ 
 Now,  $\text{im}(\widetilde h)$ is a projective $A/(c)$-module of rank one and  generated by a single element. By \cite[Theorem 2.2.15]{ir}, we conclude that 
 $$\text{im}(\widetilde h)\simeq{A/(c)}.$$
 Therefore,  
 $$\text{im}(\widetilde h)^* = \text{Hom}_{A/(c)}(A/(c),A/(c)) \simeq{A/(c)},$$ 
 which implies that $J/(c)$ is cyclic.  Thus $J$ is generated by two elements. 
 
 Since $J$ is a local complete intersection ideal of height $2$,  Lemma \ref{fact-1} shows that  $J$ is in fact a complete intersection ideal. Hence  $I$ is a set-theoretic complete intersection.\qed
\end{proof}
\medskip

The next corollary  follows directly from Theorem \ref{main1}. 
\bc\label{set-coro}
Let $A$ be a ring of dimension $3$, and let $\mathcal{C}\subset {\rm Spec}(A)$ be a curve which is local complete intersection. Then $\mathcal{C}$ is a set-theoretic complete intersection.
\ec
\medskip


A closer look at the proof of Theorem \ref{main1} reveals that the same method yields the following conclusion.

\bt\label{main3}
Let $A$ be a ring of dimension $\geq 2$, and let $I\subset A$ be a local complete intersection ideal of height $2$ such that $I/I^2$ is a free $A/I$-module of rank $2$. Then $I$ is a complete intersection. 
\et
\smallskip

As direct consequences of Theorem \ref{main3},  we obtain the next two corollaries.

\bc\label{curve-ci}
Let $A$ be a ring of dimension $3$,  and let $\mathcal{C}\subset {\rm Spec}(A)$ be a local complete intersection curve with trivial conormal bundle. Then $\mathcal{C}$ is a  complete intersection.
\ec
\medskip

\bc\label{surface-ci}
Let $A$ be a ring of dimension $4$,   and let $\mathcal{S}\subset {\rm Spec}(A)$ be a local complete intersection surface with trivial conormal bundle.  Then $\mathcal{S}$ is a  complete intersection.
\ec
\medskip

Next, we turn our attention to local complete intersection surfaces inside an affine algebra of dimension $4$ over the algebraic closure of a finite field.

Let $p$ be a prime and $\ol{\mathbb{F}}_p$ be the algebraic closure of the finite field of $p$ elements.\\
A refinement of Ferrand's construction \cite{fe}  in codimension two, for affine algebras over $\ol{\mathbb{F}}_p$, is given in  \cite[Theorem 4.1]{d1}.

\begin{theorem}\label{fe}
Let $A$ be an affine algebra over $\ol{\mathbb{F}}_p$ and $I\subset A$ be a local complete intersection ideal of height $r\geq 2$ such that ${\rm dim}(A/I)\leq 2$. Then there exists an ideal  $J\subset A$ satisfying:
\begin{enumerate}
 \item $\sqrt{J}= \sqrt{I}$.
 \item $J$ is a local complete intersection ideal of height $r$.
 \item $J/J^2$ is a free $A/J$-module of rank $r$.
\end{enumerate}
\end{theorem}
 \medskip
 

We are now prepared to establish the result that extends \cite[Remark 1.16]{mu0} (see also \cite[Theorem 1.7]{bms}) to  four-dimensional affine algebras over $\ol{\mathbb{F}}_p$. 

\begin{theorem}\label{main2}
  Let $A$ be an affine algebra  of dimemsion $4$ over $\ol{\mathbb{F}}_p$. Let $I\subset A$ be a local complete intersection  ideal of height $2$. Then $I$ is a set-theoretic complete intersection. 
  \end{theorem}
\proof
The argument proceeds along the same lines as the proof of Theorem \ref{main1}.  The only modification required is that, in the present situation, we apply Theorem \ref{fe} in place of Theorem \ref{fs}, which was invoked in the proof of Theorem \ref{main1}.
\qed
\medskip

\bc\label{surface-sci}
Let $A$ be an affine algebra  of dimemsion $4$ over $\ol{\mathbb{F}}_p$, and let $\mathcal{S}\subset {\rm Spec}(A)$ be a local complete intersection surface. Then $\mathcal{S}$ is a  set-theoretic complete intersection.
\ec
\bigskip


 \noindent
{\bf Acknowledgement :}
The first named author (Reference number 231610213879)  acknowledges University Grants Commission (UGC) for their support.

\end{document}